\documentclass[a4paper,12pt]{article}
\usepackage[margin=2cm]{geometry}

\usepackage{amssymb, amsmath, latexsym,color}
\usepackage{amsthm}
\usepackage{mathtools, mathrsfs}
\usepackage{color}
\usepackage{epsfig}
\usepackage{latexsym}
\usepackage{graphicx}
\usepackage{caption}
\usepackage{subcaption}
\usepackage{bm}
\usepackage{tcolorbox}
\usepackage{hyperref}

\newtheorem{theorem}{Theorem}[section]
\newtheorem{lemma}[theorem]{Lemma}

\newtheorem{corollary}[theorem]{Corollary}

\newtheorem{example}[theorem]{Example}

\renewcommand{\emptyset}{\varnothing}

\renewcommand{\leq}{\leqslant}
\renewcommand{\ge}{\geqslant}
\renewcommand{\le}{\leqslant}

\def\F{\mathcal{F}}
\def\Z{\mathbb{Z}}

\def\eref#1{$(\ref{#1})$}
\def\sref#1{\S$\ref{#1}$}
\def\lref#1{Lemma~$\ref{#1}$}
\def\tref#1{Theorem~$\ref{#1}$}

\def\cref#1{Conjecture~$\ref{#1}$}
\def\Cref#1{Corollary~$\ref{#1}$}

\begin{document}

\title{Mutually orthogonal binary\\ frequency squares of mixed type}

\author{Carly Bodkin and Ian M. Wanless\\
  \small School of Mathematics\\[-0.75ex]
  \small Monash University\\[-0.75ex]
  \small Clayton Vic 3800 Australia\\
\small \texttt{\{carly.bodkin,ian.wanless\}\ @monash.edu}}

\date{}

\maketitle

\abstract{A \emph{frequency square} is a matrix in which each row and
  column is a permutation of the same multiset of symbols. Two
  frequency squares $F_1$ and $F_2$ with symbol multisets $M_1$ and
  $M_2$ are \emph{orthogonal} if the multiset of pairs obtained by
  superimposing $F_1$ and $F_2$ is $M_1\times M_2$. A set of MOFS is a
  set of frequency squares in which each pair is orthogonal. We first
  generalise the classical bound on the cardinality of a set of MOFS
  to cover the case of \emph{mixed type}, meaning that the symbol
  multisets are allowed to vary between the squares in the set.

  A frequency square is \emph{binary} if it only uses the symbols 0
  and 1.  We say that a set $\F$ of MOFS is \emph{type-maximal} if it
  cannot be extended to a larger set of MOFS by adding a square whose
  symbol multiset matches that of at least one square already in $\F$.
  Building on pioneering work by Stinson, several recent papers have
  found conditions that are sufficient to show that a set of binary
  MOFS is type-maximal. We generalise these papers in several
  directions, finding new conditions that imply type-maximality. Our
  results cover sets of binary frequency squares of mixed type. Also,
  where previous papers used parity arguments, we show the merit of
  arguments that use moduli greater than 2.  }

\section{Introduction}

A \emph{frequency square} of {\it type} $(n; \lambda_0,\lambda_1,\dots, \lambda_{m-1})$ is an $n \times n$ array with entries from the set $\{0,1,\dots,m-1\}$, where entry $i$ occurs $\lambda_i$ times in every row and $\lambda_i$ times in every column. We say that symbol $i$ has \emph{frequency} $\lambda_i$ and note that $n=\sum_{i=0}^{m-1} \lambda_i$. Let $F_1$ and $F_2$ be frequency squares of type $(n; \lambda_0,\lambda_1,\dots, \lambda_{m_1-1})$ and $(n; \mu_0,\mu_1,\dots, \mu_{m_2-1})$ respectively. Then $F_1$ and $F_2$ are {\it mutually orthogonal} if, when superimposed, each of the $m_1m_2$ possible ordered pairs $(i,j)$ with $0\leq i \leq m_1-1$ and $0 \leq j \leq m_2-1$, occurs exactly $\lambda_i \mu_j$ times. A set of frequency squares is said to be {\it mutually orthogonal} if every two distinct members of the set are orthogonal. A set of $k$ mutually orthogonal frequency squares (MOFS) of order $n$ will be referred to as a set of $k$-MOFS$(n)$, or simply a set of $k$-MOFS. If we do not require that every square in a set of MOFS has the same type then we say that the MOFS have \emph{mixed type}. Since the unique frequency square of type $(n;n)$ is trivially orthogonal to every other frequency square, we will assume all our frequency squares have at least two symbols.

A set $\{F_1,F_2,\dots,F_k\}$ of MOFS is said to be \emph{maximal} if there does not exist a frequency square $F$ that is orthogonal to $F_i$ for $1 \leq i \leq k$. Note that this definition of maximal does not require $F$ to be of any particular type. We also define a more restricted version of maximality. 
A set $\{F_1,F_2,\dots,F_k\}$ of MOFS is said to be \emph{type-maximal} if there does not exist a frequency square $F$ of type $(n; \lambda_0,\lambda_1,\dots, \lambda_{m-1})$ such that
\begin{itemize}
\item $F$ is orthogonal to $F_i$ for each $1 \leq i \leq k$, and
\item there is some $t \in \{1,2,\dots,k\}$ such that $F_t$ is also of type $(n; \lambda_0,\lambda_1,\dots, \lambda_{m-1})$.
\end{itemize}
A set of MOFS that is maximal is necessarily type-maximal. The converse statement fails, as will be shown by Example \ref{not-max} in \sref{rel-max}.  To date, research into the maximality of MOFS has focused primarily on type-maximality (see \cite{MOFS-2019,MOFS-2020,JJ-2020}). However, those papers do not consider MOFS of mixed type, and they use ``maximal'' in place of what we are calling ``type-maximal''.

We mainly focus on sets of \emph{binary} MOFS, where each of the frequency squares has two symbols (but may be of mixed type, unless specified). Sometimes we will require all frequency squares in a set to be of the same type, and in this case we use the notation $k$-MOFS$(n; \lambda_0,\lambda_1,\dots, \lambda_{m-1})$ to denote a set of $k$-MOFS in which each frequency square is of type $(n; \lambda_0,\lambda_1,\dots, \lambda_{m-1})$. Note that applying a surjection onto $\{0,1\}$ (or, indeed, any other function) to the symbols of a frequency square does not alter orthogonality between that square and any other frequency square. For this reason, a set of MOFS is maximal if and only if it cannot be extended by a binary frequency square. 

A binary $k$-MOFS$(n;n-1,1)$ is equivalent to an {\it equidistant permutation array} (EPA). An EPA $A(n,d;k)$ is a $k \times n$ array in which each row contains each integer from $1$ to $n$ precisely once, and every pair of rows differ in precisely $d$ positions. Now, each square of a $k$-MOFS$(n;n-1,1)$ is a permutation matrix. Taking the corresponding permutations as the rows of a $k \times n$ array yields an $A(n,n-1;k)$. Hence, our results in \sref{comp} regarding maximal $k$-MOFS$(n;n-1,1)$ duplicate results already published in the literature on EPAs (see \cite{HCD}).

The structure of this paper is as follows. In \sref{s:compsets} we generalise the classical upper bound on the cardinality of a set of MOFS, to the case of mixed type. In \sref{Relations} we introduce the idea of relations, which originated in a seminal paper by Stinson \cite{LS}. We find conditions that must be satisfied by an arbitrary (that is, not necessarily mutually orthogonal) set of frequency squares that satisfy a relation. In \sref{s:relbinMOFS} we specialise to study relations on binary MOFS. Our results in this section generalise theorems from \cite{MOFS-2019} and \cite{JJ-2020} to the case of MOFS of mixed type. In \sref{rel-max} we establish several conditions that can be used to show type-maximality. In fact our conditions usually demonstrate that a set of MOFS cannot be extended by adding a new frequency square unless every symbol has even multiplicity in that square. Finally, in \sref{comp} we look at some data from orders $n\le6$. We consider how large a set of MOFS can be if we allow its squares to have certain types. This data provides several examples that illustrate the results in earlier sections. In some cases those examples where the inspiration for the results themselves.

\section{Complete sets of MOFS}\label{s:compsets}

Most research into sets of MOFS has focused on squares of type $(n; \lambda_0,\lambda_1,\dots, \lambda_{m-1})$ where $\lambda_0=\lambda_1=\dots=\lambda_{m-1}=\lambda$. We will refer to these as MOFS of type $(n;\lambda)$ and describe them as \emph{balanced}. We typically do \emph{not} assume that our frequency squares are balanced. Indeed, we do not even assume that each frequency square in a set has the same type.

It is well known that in the balanced case, the maximum possible cardinality of a set of MOFS is $(n-1)^2/(m-1)$. This was recently re-proved by Cavenagh, Mammoliti and Wanless \cite{MOFS-2020}, and we generalise their proof to include sets of MOFS of mixed type. 

\begin{theorem}\label{gen-bound}
Let $\F=\{F_1,F_2,\dots,F_k\}$ be a set of {\rm MOFS}$(n)$ and let $F_t$ have symbol set $\{0,1,\dots,m_t-1\}$ for $1\leq t \leq k$. Then we have
\begin{equation}\label{upper-bd}
\sum_{i=1}^k (m_i-1) \leq (n-1)^2.
\end{equation}
\end{theorem}

\begin{proof}
Suppose that $F_t$ has type $(n; \lambda_{0,t},\lambda_{1,t},\dots, \lambda_{m_t-1,t})$ for $1\leq t \leq k$. We consider $n \times n$ matrices as vectors in an $n^2$-dimensional vector space equipped with the inner product $A \circ B= \sum_i \sum_j a_{ij}b_{ij}$ for matrices $A=[a_{ij}]$ and $B=[b_{ij}]$. We define the following matrices for $1\leq r,c \leq n$, $0 \leq s < m_t$ and $1\leq t \leq k$
\begin{align*}
	\mathscr{R}_r[i,j]& =
	\begin{cases}
		1 & \text{if } i=r \text{ and} \\
		0 & \text{otherwise,}
	\end{cases} \\
	\mathscr{C}_c[i,j]&=
	\begin{cases}
		1 & \text{if } j=c \text{ and} \\
		0 & \text{otherwise,}
	\end{cases} \\
	\mathscr{I}_{s,t}[i,j]&=
	\begin{cases} 
		1 & \text{if } F_t[i,j]=s \text{ and} \\
		0 & \text{otherwise.}
	\end{cases} 
\end{align*}

We claim that the following is a linearly independent set, where $J_n$ is the all ones matrix of order $n$:
\begin{equation} \label{ind-set}
\{J_n\} \cup \{\mathscr{R}_r : 1 \leq r \leq n-1\} \cup \{\mathscr{C}_c : 1 \leq c \leq n-1\} \cup \{\mathscr{I}_{s,t}: 1 \leq s < m_t, 1\leq t \leq k\}.
\end{equation} 

This set has cardinality $1+2(n-1)+\sum_{i=1}^k (m_i-1)$ and as we are in an $n^2$-dimensional vector space, the result will follow from our claim.
For each $r,c,s$ and $t$ we define the new matrices $\mathscr{R}'_r=n\mathscr{R}_r-J_n$ and $\mathscr{C}'_c=n\mathscr{C}_c-J_n$ and 
$\mathscr{I}'_{s,t}=(n/\lambda_{s,t})\mathscr{I}_{s,t}-J_n$. It is trivial that the union of $\{J_n\}$ with either $\{\mathscr{R}_r : 1 \leq r \leq n-1\}$ or $\{\mathscr{C}_c :1 \leq c \leq n-1\}$ yields a linearly independent set. Furthermore, for any given $t$, the set $\{J_n\} \cup \{\mathscr{I}_{s,t} : 1 \leq s < m_t, 1\leq t \leq k\}$ is also linearly independent. It is straightforward to show the following holds
$$\mathscr{R}'_r \circ \mathscr{C}'_c = \mathscr{R}'_r \circ \mathscr{I}'_{s,t}=\mathscr{C}'_c \circ \mathscr{I}'_{s,t}= \mathscr{I}'_{s,t} \circ \mathscr{I}'_{s',t'}=0,$$
for all $r,c,s,t,s',t'$ provided $t\neq t'$. Hence (\ref{ind-set}) is a linearly independent set and we are done.
\end{proof}

 A set of MOFS that achieves (\ref{upper-bd}) will be referred to as {\it complete}. Note that a complete set of MOFS is trivially maximal. If we consider the case where all frequency squares in a set of MOFS have the same number of symbols (but can still be of any type), we see immediately that the maximum number is the same as for MOFS of type $(n; \lambda)$.
\begin{corollary}\label{cor-gen-bound}
Let $\F$ be a set of $k$-{\rm MOFS}$(n)$. If each of the frequency squares in $\F$ has precisely $m$ symbols then we have
$$k \leq \dfrac{(n-1)^2}{(m-1)}.$$
\end{corollary}

In \cite{MOFS-2019} it was shown that there are no complete sets of type $(n;n/2)$ when $n\equiv2\bmod4$. However, that result does not rule out complete sets of unbalanced frequency squares or sets of mixed type.

\section{Relations on arbitrary sets of frequency squares}\label{Relations}

Relations were first used in \cite{LS} to demonstrate the maximality of sets of mutually orthogonal Latin squares. The idea of relations in the context of other frequency squares was used extensively throughout \cite{MOFS-2019} and \cite{JJ-2020}. We generalise those works in several ways. In this section, we look at sets of frequency squares which are not necessarily mutually orthogonal. We also allow the squares to have different frequencies (i.e.~to be of mixed type).

A set $\F=\{F_1,F_2, \dots, F_k\}$ of frequency squares of order $n$ and of mixed type, may be written as an $n^2 \times (k+2)$ array, $\mathcal{O}$ with rows given by
\begin{equation}\label{array-row}
\big[i,j,F_1[i,j],F_2[i,j], \dots, F_k[i,j] \big],
\end{equation}
for $1\leq i,j \leq n$. We will refer to $\mathcal{O}$ as the \emph{array corresponding to $\F$}. In order for the columns of $\mathcal{O}$ to be well-defined, we should consider $\F$ to have an indexing that implies an ordering on the frequency squares. Let $Y_c$ be the set of symbols that occur in column $c$ of $\mathcal{O}$. Then a \emph{relation} is a $(k+2)$-tuple $\mathcal{R}=(X_1,X_2, \dots, X_{k+2})$ of sets such that $X_i \subseteq Y_i$ for $1 \leq i \leq k+2$, with the property that every row (\ref{array-row}) of $\mathcal{O}$ has an even number of columns $c$ for which the symbol in column $c$ is a element of $X_c$. In other words, $\big|\{c : \mathcal{O}[r,c] \in X_c, 1 \leq c \leq k+2\}\big|\equiv0\bmod2$ for every $r \in \{1, \dots, n^2\}$. We say that $\F$ \emph{satisfies} the relation $\mathcal{R}$.

A relation is \emph{trivial on column c} if $X_c = \emptyset$ or $X_c=Y_c$. We will say that a relation is \emph{non-trivial} if it is not trivial on at least one column. Furthermore, we say a relation is \emph{full} if it is non-trivial on every column except possibly the first two. A relation for which $\{|X_1|,|X_2|\}\subseteq\{0,n\}$ is called {\it constant}. We are most often interested in relations that are not constant, which we will refer to as \emph{non-constant}.

An important observation about relations is the following: if we start with a relation $(X_1, X_2 \dots, X_{k+2})$ then we may swap any pair of distinct elements $(X_i,X_j)$ with their complements $(X_i^{\mathsf{c}},X_j^{\mathsf{c}})$ and the resultant $(k+2)$-tuple is still a relation. This observation has a particularly useful implication for binary frequency squares. Suppose that a set of binary frequency squares satisfies a relation $(X_1, X_2 \dots, X_{k+2})$. Then $X_c \subseteq \{0,1\}$ for $3 \leq c \leq k+2$. By complementing the pair $(X_1,X_c)$ as necessary, we may assume $X_c \subseteq \{1\}$ for $3 \leq c \leq k+2$. This makes relations in the binary case easy to grasp, as is illustrated by the following elegant characterisation from \cite{MOFS-2019} and \cite{JJ-2020}. We state and prove it in a slightly more general setting.

\begin{lemma}\label{lem-MOFS-2019}
A set $\F$ of binary frequency squares satisfies a non-trivial relation if and only if some non-empty subset of $\F$ has a $\Z_2$-sum that, up to permutation of the rows and columns, has the following structure of constant blocks.
\begin{equation}\label{block-matrix}
 \left[ \begin{array}{cc}
\bm{0} & \bm{1} \\
\bm{1} & \bm{0}
\end{array}\right].
\end{equation}
\end{lemma}

\begin{proof}
Let $\F=\{F_1,F_2,\dots,F_k\}$ be a set of binary frequency squares. Firstly, suppose that $\F$ satisfies a full relation $\mathcal{R}=(X_1,X_2,\dots,X_{k+2})$. For $r,c \in \{1,2,\dots,n\}$, let $x_{r,c}$ denote the $\Z_2$-sum over the entries in the cell $(r,c)$ of each of the squares in $\F$. That is, $x_{r,c} \equiv\sum_{t=1}^k F_t[r,c] \bmod 2$. By the definition of a relation, $x_{r,c}=1$ if precisely one of $r \in X_1$ or $c \in X_2$ holds, and $x_{r,c}=0$ otherwise. Therefore, by permuting the rows and columns of $\F$ such that $X_1=\{1,2,\dots,|X_1|\}$ and $X_2=\{1,2,\dots,|X_2|\}$, the $\Z_2$-sum of $\F$ will have the structure in (\ref{block-matrix}).

Secondly, suppose the $\Z_2$-sum of $\F$ has the structure in (\ref{block-matrix}), up to permutation of the rows and columns. Let $X_1$ and $X_2$ contain, respectively, the rows and the columns that induce the upper left block of zeros in (\ref{block-matrix}). Then it is easily checked that $\F$ satisfies the full relation $(X_1,X_2,\dots,X_{k+2})$ where $X_c=\{1\}$ for $3\leq c \leq k+2$.

We have shown that a set $\F$ of binary frequency squares satisfies a full relation if and only if the $\Z_2$-sum of $\F$ has the structure in (\ref{block-matrix}), up to permutation of rows and columns.
Lastly, it is clear that a set of frequency squares satisfies a non-trivial relation if and only if it contains a non-empty subset that satisfies a full relation. This completes the proof.
\end{proof}

Let $\mathcal{R}=(X_1,X_2,\dots,X_{k+2})$ be a relation with $|X_1|=a$ and $|X_2|=b$ for some $a,b \in \{0,\dots,n\}$. The block structure in (\ref{block-matrix}) demonstrates that when considering relations in the binary case, we only really care about the number of rows involved (the cardinality of $X_1$) and the number of columns involved (the cardinality of $X_2$). Therefore, for simplicity, we will refer to $\mathcal{R}$ as an $(a,b)$-relation. \lref{lem-MOFS-2019} implies that a set $\F$ of binary frequency squares satisfies a constant relation if and only if some non-empty subset of $\F$ has $\Z_2$-sum that is either the all zeros or all ones matrix (see Examples \ref{EJCgoof} and \ref{constant-rel}). Since we may swap the pair $(X_1,X_2)$ for $(X_1^{\mathsf{c}},X_2^{\mathsf{c}})$ in the relation, an $(a,b)$-relation is considered equivalent to an $(n-a,n-b)$-relation. Hence, each constant relation is equivalent to either an $(n,n)$-relation or an $(n,0)$-relation.

It is important to stress that we allow some of the blocks in \eref{block-matrix} to be degenerate. This will certainly be the case when we have a constant relation, such as in the following example. In this and many subsequent examples, we will present our {\rm MOFS} in superimposed format.

\begin{example}\label{EJCgoof}
Consider the following set of binary $4$-{\rm MOFS}$(8;4,4):$
\[
\left[
\begin{array}{cccccccc}
1111&0000&0011&1100&1111&0101&0000&1010\\
0101&1010&0110&1001&0000&1010&1111&0101\\
0000&1111&1100&0011&0011&0110&1100&1001\\
1010&0101&1001&0110&1100&1001&0011&0110\\
1111&0101&0000&1010&1111&0000&0011&1100\\
0000&1010&1111&0101&0101&1010&0110&1001\\
0011&0110&1100&1001&0000&1111&1100&0011\\
1100&1001&0011&0110&1010&0101&1001&0110\\
\end{array}\right].
\]
These {\rm MOFS} satisfy a non-trivial constant full $(8,8)$-relation, since their $\Z_2$-sum is
the zero matrix. No proper subset of these {\rm MOFS} satisfies a relation.
\end{example}

We have made a subtle change to the definition of a full relation,
compared to the definition given in \cite{MOFS-2019}. The difference
is that we are allowing full relations to be trivial on both of the
first two columns, as happens in Example~\ref{EJCgoof}. Failure to do so
in \cite{MOFS-2019} led to an erroneous claim  that a set of MOFS
satisfies a non-trivial relation if and only if a non-empty subset
satisfies a full relation.  This claim is true with our definition of full
but not with the definition used in \cite{MOFS-2019}, with
Example~\ref{EJCgoof} providing a counterexample.

\begin{example}\label{ex1}
The following is an example of a set of $2$-{\rm MOFS}$(6;2,4)$ that satisfies a full $(4,4)$-relation:
\[ 
\left[\begin{array}{cccccc}
0&1&1&1&0&1\\
1&0&1&1&0&1\\
1&1&0&1&1&0\\
1&1&1&0&1&0\\
0&0&1&1&1&1\\
1&1&0&0&1&1\end{array} \right]
+
\left[ \begin{array}{cccccc}
0&1&1&1&1&0\\
1&0&1&1&1&0\\
1&1&0&1&0&1\\
1&1&1&0&0&1\\
1&1&0&0&1&1\\
0&0&1&1&1&1\end{array} \right]
\equiv
\left[ \begin{array}{cccccc}
0&0&0&0&1&1\\
0&0&0&0&1&1\\
0&0&0&0&1&1\\
0&0&0&0&1&1\\
1&1&1&1&0&0\\
1&1&1&1&0&0\end{array} \right].
\]
In particular $X_1=X_2=\{1,2,3,4\}$. We also give its $\Z_2$-sum. 
\end{example}

An interesting feature of Example~\ref{ex1} is that it shows a full relation being satisfied by an even number of binary {\rm MOFS} of order $2\bmod4$. Such a thing is impossible when the {\rm MOFS} are balanced, as shown by Theorem $5$ in {\rm\cite{MOFS-2019}} or Corollary 15 in \cite{JJ-2020}. This provides some motivation to see whether we can generalise results from \cite{MOFS-2019} and \cite{JJ-2020}, which were only concerned with balanced MOFS. The next result does just that.

\begin{theorem}\label{thm-rel-gen}
Let $\F$ be a set of $k$ frequency squares of order $n$ that satisfies a non-constant relation $(X_1,X_2,\dots,X_{k+2})$. Then we have 
\begin{itemize}
\item[(i)] $n$ is even, and 
\item[(ii)] $|X_1| \equiv |X_2| \bmod 2.$
\end{itemize}
\end{theorem}

\begin{proof}
  Let $\F=\{F_1,\dots,F_k\}$ and let $\mathcal{O}$ be the array corresponding to $\F$.
  We first argue that we may assume that $\mathcal{R}=(X_1,X_2,\dots,X_{k+2})$ is a full relation. If $X_c=\emptyset$ for some $c\ge3$ then we can simply remove $F_{c-2}$ from $\F$. And if $X_c=Y_c$ for some $c\ge3$ then we can complement $X_c$ and some $X_{c'}$ for $c'\in\{3,\dots,k\}\setminus\{c\}$, before removing $F_{c-2}$ from $\F$. In this way we will reach a set of MOFS that satisfies a non-constant full relation, without changing $n$, $X_1$ or $X_2$. Note that a non-constant relation has to be non-trivial on at least one column other than the first two, which means we always have an option to choose $X_{c'}$.

  Since $\mathcal{R}$ is non-constant, it cannot be trivial on {\it both} of the first two columns of $\mathcal{O}$. Furthermore, if one of $X_1$ or $X_2$ is empty, we may swap the pair $(X_1,X_2)$ with $(X_1^{\mathsf{c}},X_2^{\mathsf{c}})$ and the resultant relation is still non-constant. So we may assume that both $X_1$ and $X_2$ are non-empty. Permute the rows and columns of the frequency squares so that $X_1=\{1,\dots,|X_1|\}$  and $X_2=\{1,\dots,|X_2|\}$. Hence, we may assume that $1 \in X_1 \cap X_2$.

Consider the rows of $\mathcal{O}$ that correspond to the first column of the frequency squares, namely, the rows for which $j=1$ in (\ref{array-row}). Let $s$ be the number of cells in these rows that contain symbols in $\mathcal{R}$. So $s=\big|\{(r,c) \in \mathcal{N} :  \mathcal{O}[r,c] \in X_c \text{ and } \mathcal{O}[r,2]=1\}\big|$ where $\mathcal{N}=[n^2] \times [k+2]$. We have
$$s = n+|X_1|+y,$$
where $y=\big|\{(r,c) \in \mathcal{N} : \mathcal{O}[r,c] \in X_c, \mathcal{O}[r,2]=1 \text{ and }3 \leq c \leq k+2\}\big|$.
By the definition of a relation, every row of $\mathcal{O}$ must have an even number of symbols in $\mathcal{R}$. Hence $s$ must be even and we have
\begin{equation}\label{rel-parity3} |X_1| \equiv n+y \mod 2. \end{equation}
We can do a similar count of the symbols in $\mathcal{R}$ along the rows of $\mathcal{O}$ that correspond to the cells in the first row of the frequency squares. This yields
\begin{equation}\label{rel-parity2} |X_2| \equiv n+x \mod 2, \end{equation}
where $x=\big|\{(r,c) \in \mathcal{N} : \mathcal{O}[r,c] \in X_c, \mathcal{O}[r,1]=1 \text{ and }3 \leq c \leq k+2\}\big|$.
However, given any frequency square in $\F$, the multi-set of symbols in its first row is precisely the same as the multi-set of symbols in its first column. So we must have $x=y$, which combines with \eref{rel-parity3} and \eref{rel-parity2} to prove \emph{(ii)}.

Since $\mathcal{R}$ is non-constant we cannot have $|X_1|=|X_2|=n$. Firstly, suppose $|X_1|<n$. Let $t=|\{(r,c) \in \mathcal{N}: \mathcal{O}[r,c] \in X_c \text{ and } \mathcal{O}[r,1]=|X_1|+1\}|$ be the number of cells in $\mathcal{O}$ that correspond to row $|X_1|+1$ of the frequency squares and contain symbols in the relation. Since $|X_1|+1\notin X_1$, we have $t=|X_2|+x$.
Again, by the definition of a relation, $t$ must be even.
Combining this with (\ref{rel-parity2}) we have
$x\equiv |X_2| \equiv x+n\bmod 2$, so $n$ is even.

In the case when $|X_1|=n$ we must have $|X_2|<n$. So we can deduce that $n$ is even using a similar argument to the above, but considering column $|X_2|+1$ of the frequency squares. This completes the proof of part \emph{(i)}.
\end{proof}

In the binary setting, we can say a little more about the parities of $|X_1|$ and $|X_2|$. Firstly we provide a simple result that we will use several times.

\begin{lemma}\label{lm:parity}
Let $\F=\{F_1,F_2,\dots,F_k\}$ be a set of binary frequency squares where $F_t$ has type $(n; \lambda_{0,t},\lambda_{1,t})$ for $1 \leq t \leq k$. Then $\sum_{t=1}^k \lambda_{0,t} \equiv \sum_{t=1}^k \lambda_{1,t} \bmod 2$ if and only if at least one of $k$ or $n$ is even.
\end{lemma}

\begin{proof}
Summing $n = \lambda_{0,t} +\lambda_{1,t}$ for $1 \leq t \leq k$ yields $kn = \sum_{t=1}^k \lambda_{0,t} + \sum_{t=1}^k \lambda_{1,t}$, and the result follows.
\end{proof}

\begin{theorem}\label{thm-rel}
Let $\F=\{F_1,F_2,\dots,F_k\}$ be a set of binary frequency squares where $F_t$ has type $(n; \lambda_{0,t},\lambda_{1,t})$ for $1 \leq t \leq k$. If $\F$ satisfies a non-constant full $(a,b)$-relation then 
$$a \equiv b \equiv \sum_{t=1}^k \lambda_{0,t} \equiv \sum_{t=1}^k \lambda_{1,t} \mod 2.$$
\end{theorem}

\begin{proof}
Let $\F=\{F_1,F_2,\dots,F_k\}$ be a set of binary frequency squares where $F_t$ has type $(n; \lambda_{0,t},\lambda_{1,t})$ for $1 \leq t \leq k$. Suppose that $\F$ satisfies a non-constant full $(a,b)$-relation $\mathcal{R}=(X_1,X_2,\dots,X_{k+2})$ and let $\mathcal{O}$ be the array corresponding to $\F$. Since we are dealing with binary frequency squares, we may assume $X_c = \{1\}$ for $3 \leq c \leq k+2$. \tref{thm-rel-gen} implies that $n$ is even, and this combined with \eref{rel-parity3} and \eref{rel-parity2} yields $a \equiv b \equiv x \bmod 2$,
where $x$ counts the total number of ones in the first rows of the frequency squares. Hence,
\begin{equation}\label{bin-rel}
a \equiv b \equiv \sum_{t=1}^k \lambda_{1,t} \mod 2.
\end{equation}
Since $n$ is even, \lref{lm:parity} implies that $\sum_{t=1}^k \lambda_{1,t}$ and $\sum_{t=1}^k \lambda_{0,t}$ have the same parity. Combining this with (\ref{bin-rel}) completes the proof.
\end{proof}

If we only allow binary frequency squares of the same type, \tref{thm-rel} implies the following.

\begin{corollary}\label{cor-rel}
Let $\F$ be a set of $k$ binary frequency squares of type $(n; \lambda_0,\lambda_1)$ that satisfies a non-constant full $(a,b)$-relation. Then 
$a \equiv b \equiv \lambda_0 k \equiv \lambda_1 k \mod 2$.
\end{corollary}

It is important to note that in Theorems \ref{thm-rel-gen} and \ref{thm-rel} it is necessary that we require the relation to be non-constant. This is illustrated in the following example, which mirrors an example in \cite{JJ-2020}.

\begin{example}\label{constant-rel}
The following is a set of three frequency squares of type $(3;2,1)$ that satisfies a constant full $(3,0)$-relation. We also give their $\Z_2$-sum, which we note is a constant block of ones:
\[ 
\left[\begin{array}{cccccc}
1&0&0\\
0&1&0\\
0&0&1\end{array} \right]
+
\left[ \begin{array}{cccccc}
0&1&0\\
0&0&1\\
1&0&0\end{array} \right]
+
\left[ \begin{array}{cccccc}
0&0&1\\
1&0&0\\
0&1&0\end{array} \right]
\equiv
\left[ \begin{array}{cccccc}
1&1&1\\
1&1&1\\
1&1&1\end{array} \right].
\]
As $n$ is odd, and $3=a \not\equiv b=0 \bmod 2$, this example fails both conclusions of \tref{thm-rel-gen}. The conclusion of \Cref{cor-rel} also fails, because $a \not\equiv b$ and $6=k\lambda_0\not\equiv k\lambda_1=3\bmod2$.
\end{example}

The following result deals with the case where the relation is constant. 

\begin{theorem}\label{thm-con-rel}
Let $\F=\{F_1,F_2,\dots,F_k\}$ be a set of binary frequency squares where $F_t$ has type $(n; \lambda_{0,t},\lambda_{1,t})$ for $1 \leq t \leq k$.
If $\F$ satisfies a constant full $(n,b)$-relation we have 
\begin{itemize}
\item[(i)] If $b=n$ then $\sum_{t=1}^k \lambda_{1,t} \equiv 0 \bmod 2$, and
\item[(ii)] If $b=0$ then $\sum_{t=1}^k \lambda_{1,t} \equiv n \bmod 2$.
\end{itemize}
\end{theorem}

\begin{proof}
Suppose $\F$ satisfies a constant full $(n,b)$-relation $\mathcal{R}=(X_1,\dots,X_{k+2})$ and let $\mathcal{O}$ be the array corresponding to $\F$. Since $|X_1|=n$, we have $1\in X_1$, so we can apply \eref{rel-parity2}. It yields that $n+b+\sum_{t=1}^k \lambda_{1,t} \equiv 0 \bmod 2$. Since $\mathcal{R}$ is constant, $b \in \{0,n\}$ and the result follows.
\end{proof}

If each square is of the same type, \tref{thm-con-rel} implies a constant full relation cannot be satisfied in some cases. 

\begin{corollary}\label{cor-con-rel}
Suppose $k$, $\lambda_0$ and $\lambda_1$ are all odd and let $\F$ be a set of $k$ binary frequency squares of type $(n;\lambda_0,\lambda_1)$. If $\F$ satisfies a full relation, then this relation must be non-constant.
\end{corollary}

\begin{proof}
Suppose that $\F$ satisfies a constant full relation. Note that $n=\lambda_0+\lambda_1\equiv0\bmod2$. So \tref{thm-con-rel} implies that $k\lambda_1=\sum_{i=1}^k \lambda_1 \equiv 0 \bmod 2$. However, this is impossible, since both $k$ and $\lambda_1$ are odd.
\end{proof}

\section{Relations on binary MOFS}\label{s:relbinMOFS}

So far, we have only presented necessary conditions for sets of frequency squares to satisfy a relation. If we also suppose the frequency squares are binary and mutually orthogonal then we have more restrictions. The results in this section are generalisations of results in \cite{MOFS-2019} and \cite{JJ-2020}.

\begin{theorem}\label{thm-rel2}
Let $\F=\{F_1,F_2, \dots, F_k\}$ be a set of binary {\rm MOFS} where $F_t$ has type $(n; \lambda_{0,t},\lambda_{1,t})$ for $1 \leq t \leq k$. If $\F$ satisfies a non-constant full relation then either
\begin{itemize}
\item[(i)]$\lambda_{0,t} \equiv \lambda_{1,t} \equiv 0\bmod2$ for all $t \in \{1,2,\dots,k\}$, or
\item[(ii)]$\lambda_{0,t} \equiv \lambda_{1,t} \equiv 1\bmod2$ for an odd number of $t \in \{1,2,\dots,k\}$.
\end{itemize}
\end{theorem}

\begin{proof}
  Suppose $\F$ satisfies a non-constant full $(a,b)$-relation $\mathcal{R}=(X_1,X_2,\dots,X_{k+2})$. Suppose the symbol $1$ has odd frequency in a non-zero even number of squares. By \tref{thm-rel}, both $a$ and $b$ must be even. Without loss of generality, let $\lambda_{1,1}$ be odd. Let $\Omega_0$ (respectively $\Omega_1$) be the set of cells $(r, c)$ for which $F_1[r, c] = 1$ and in which the superposition of $F_1,\dots,F_k$ has an even (respectively, odd) number of ones. By the definition of a relation, $\Omega_1$ counts those cell $(r,c)$ for which $F_1[r,c]=1$ and precisely one of $r \in X_1$ or $c \in X_2$ holds. We claim that $|\Omega_1|$ is even, since it can be obtained by counting the (even) number of ones in the rows of $ F_1$ with indices in $X_1$, adding the (even) number of ones in the columns of $F_1$ with indices in $X_2$, and subtracting twice the number of ones in the intersection. Clearly we have $|\Omega_0|+|\Omega_1|=n \lambda_{1,1}$. Note that $n$ must be even by \tref{thm-rel-gen}(i), since $\F$ satisfies a non-constant full relation. Hence $|\Omega_0|$ is also even.

Now let $p$ be the total number of pairs of ones in the superposition of $F_1$ with the other $(k -1)$ squares. For $2\leq t \leq k$ each square $F_t$ contributes $\lambda_{1,1}\lambda_{1,t}$ to $p$, so we have
$$p=\lambda_{1,1} \sum_{t=2}^k \lambda_{1,t}.$$

By our assumption, $\lambda_{1,t}$ is odd for an odd number of $t \in \{2,\dots,k\}$ and therefore $p$ is odd. However, each cell in $\Omega_0$ contributes an odd number of times to $p$ and each cell in $\Omega_1$ contributes an even number, showing that $p\equiv |\Omega_0| \equiv 0 \bmod 2$. This contradiction proves the result holds for the symbol $1$. Furthermore, since $n$ is even, for each $1\leq t \leq k$ we have $\lambda_{1,t} \equiv \lambda_{0,t} \bmod 2$, and we are done.
\end{proof}

If we consider the case where all squares have the same type, \tref{thm-rel2} implies the following result.

\begin{corollary}\label{cor-rel2}
Suppose $\lambda_0$ and $\lambda_1$ are odd and let $\F$ be a set of $k$-{\rm MOFS}$(n;\lambda_0,\lambda_1)$ that satisfies a non-constant full relation. Then $k$ must be odd.
\end{corollary}

It turns out we can say a little more about $k$, by working modulo $4$.

\begin{theorem}\label{thm-rel1}
Suppose $\lambda_0$ and $\lambda_1$ are odd, and let $\F$ be a set of $k$-{\rm MOFS}$(n;\lambda_0,\lambda_1)$ that satisfies a non-constant full relation. Then $k \equiv \lambda_0\lambda_1 \bmod 4$.
\end{theorem}

\begin{proof}
  Suppose $\F$ is a set of $k$-MOFS$(n;\lambda_0,\lambda_1)$ that satisfies a non-constant full relation $(X_1,X_2, \dots X_{k+2})$.
  Consider superimposing the $k$-MOFS, and define $t_i$ to be the number of the resulting $k$-tuples that contain precisely $i$ ones. Counting the total number of ones in all $k$ squares in two ways yields
\begin{equation}\label{rel1} kn\lambda_1 = \sum_{i=0}^k it_i. \end{equation}
Superimposing pairs of squares in $\F$, we can also count the number of $(1,1)$ pairs in two ways. Since the squares are mutually orthogonal, there must be $\lambda_1^2$ such pairs for each pair of squares, giving
\begin{equation}\label{rel2} {k \choose 2} \lambda_1^2 = \sum_{i=0}^{k} {i\choose2} t_i.
\end{equation}
Doubling  (\ref{rel2}) and adding (\ref{rel1}) yields
$$k(k-1)\lambda_1^2+kn\lambda_1 = \sum_{i=0}^k i^2t_1.$$
Since $\lambda_1^2 \equiv 1 \bmod 4$, $i^2\equiv1 \bmod 4$ for all odd $i$, and $i^2\equiv0 \bmod 4$ for all even $i$, we have
\begin{equation}\label{rel3} k(k-1) +kn\lambda_1 \equiv \sum_{\text{odd }i} t_i \bmod 4.
\end{equation}
The right hand side of (\ref{rel3}) gives the number of $k$-tuples with an odd number of ones. Since $\F$ satisfies the relation $(X_1,X_2,\dots,X_{k+2})$, these tuples must correspond precisely to cells $(r,c)$ of the MOFS for which exactly one of $r \in X_1$ or $c \in X_2$ holds. The number of these cells is $|X_2|(n-|X_1|)+|X_1|(n-|X_2|)$. Equating this to \eref{rel3} gives
\begin{equation}\label{rel4}
  k(k-1) +kn\lambda_1 \equiv n(|X_1|+|X_2|)-2|X_1||X_2| \mod 4.
\end{equation}
Now $n$ is even and \Cref{cor-rel} implies $|X_1|\equiv |X_2| \equiv k \bmod 2$, so $n(|X_1|+|X_2|)$ is divisible by $4$. Substituting $n=\lambda_0+\lambda_1$ into \eref{rel4} and simplifying, we find
$$ k(k+\lambda_0\lambda_1) \equiv 2k^2 \mod 4.$$
\Cref{cor-rel2} says that $k$ is odd, so this last equation simplifies to
$k \equiv \lambda_0\lambda_1 \bmod 4$.
\end{proof}

\section{Conditions implying type-maximality}\label{rel-max}

In order for the results of the previous section to be useful to us, we need to understand the connection between relations and maximality. More specifically, we would like to know under what conditions a set of MOFS that satisfies a relation can be extended to a larger set. In \cite{MOFS-2019} and \cite{JJ-2020} the authors proved that relations can be used to certify that sets of balanced MOFS are type-maximal:

\begin{theorem}\label{max-MOFS-2019}
Suppose $k$ and $\lambda$ are both odd. Let $\F$ be a set of $k$-{\rm MOFS}$(n; \lambda)$ that satisfies a non-constant full relation. Then $\F$ is type-maximal.
\end{theorem}

Our goal in this section is to explore possible generalisations of \tref{max-MOFS-2019}. We do this in several ways. Firstly, we consider sets of MOFS which are of mixed type. Perhaps more interestingly, we generalise the idea of a relation. Applications of relations essentially boil down to employing parity arguments. We will show that it can sometimes be useful to work with moduli that are larger than 2.

Prior work has mostly considered what we are calling type-maximality.
It is time to illustrate that type-maximality and maximality are
indeed different properties. 

\begin{example}\label{notmax} 
Consider the following $5$-{\rm MOFS}$(6;3,3)$, which was shown in {\rm \cite{MOFS-2019}} to be type-maximal:
\begin{equation}\label{5-MOFS(6)}
	\left[\begin{array}{cccccc}
		11111&10001&10010&00000&01111&01100\\
		10001&10101&11110&00110&01001&01010\\
		10010&11110&01000&11101&00011&00101\\
		00000&00110&11101&00111&11000&11011\\
		01111&01001&00011&11000&10110&10100\\
		01100&01010&00101&11011&10100&10011
	\end{array} \right].
\end{equation}
Note that ${\rm (\ref{5-MOFS(6)})}$ does not satisfy a full relation, since its $\Z_2$-sum is the identity matrix. Furthermore, this set of {\rm MOFS} is not maximal, since for example, it is orthogonal to the following $2$-{\rm MOFS}$(6;5,1)$:
\begin{equation}\label{not-max}
	\left[\begin{array}{cccccc}
		00&10&00&00&01&00\\
		01&0	0&00&10&00&00\\
		00&00&11&00&00&00\\
		00&01&00&00&00&10\\
		10&00&00&00&00&01\\
		00&00&00&01&10&00
	\end{array} \right].
\end{equation}
Combining ${\rm (\ref{5-MOFS(6)})}$ and ${\rm (\ref{not-max})}$ gives a binary $7$-{\rm MOFS}$(6)$, containing frequency squares of two different types.
\end{example}

In the following result, we start with a set of binary MOFS of the same type, and examine the conditions under which it can be extended by a frequency square of any type. Rather than focusing solely on the block structure in (\ref{block-matrix}), we state things more generally for any nice block structure. Let $J_{c,d}$ denote a $c\times d$ block of ones.

\begin{theorem}\label{thm-block-structure}
  Let $w$ be a non-negative integer.
Let $\F=\{F_1,F_2,\dots,F_k\}$ be a set of binary {\rm MOFS} where $F_t$ has type $(n; \lambda_{0,t},\lambda_{1,t})$ for $1 \leq t \leq k$. Suppose $\F$ has $\Z_w$-sum that, up to permutation of the rows and columns, has the following structure of constant blocks
\begin{equation}\label{block-matrix3}
 \left[ \begin{array}{cc}
{x_1}J_{a,b} & {x_2}J_{a,n-b} \\
{x_3}J_{n-a,b} & {x_4}J_{n-a,n-b}
\end{array}\right],
\end{equation}
where $x_1,x_2,x_3,x_4$ are non-negative integers satisfying $x_1+x_4\equiv x_2+x_3 \bmod w$. Let $F$ be a square of type $(n; \mu_0, \mu_1,\dots,\mu_{m-1})$. If  $\F \cup \{F\}$ is a set of $(k+1)$-{\rm MOFS} then for $0 \leq i \leq m-1$ we have
\begin{equation*}
\mu_i \sum_{t=1}^k \lambda_{1,t}\equiv a\mu_i(x_2-x_4)+b\mu_i(x_3-x_4) + \mu_i n x_4 \mod w.
\end{equation*}
\end{theorem}

\begin{proof}
  Suppose we can extend $\F$ to a $(k+1)$-MOFS$(n)$ by a new square $F$ of type $(n; \mu_0, \mu_1,\dots,\mu_{m-1})$. Write
\[
F= \left[ \begin{array}{cc}
     Q_1&Q_2\\
     Q3&Q_4
\end{array}\right],
\]
where the four subarrays of $F$ correspond to the blocks of (\ref{block-matrix3}). 
Fix a symbol $i$ in $F$. Define $q_j$ to be the number of copies of $i$ in $F$ in the subarray $Q_j$ for $1\le j\le4$. Since $F$ contains precisely $\mu_i$ copies of $i$ in every row and column, we have $q_2=a\mu_i-q_1$ and $q_3=b\mu_i-q_1$. Since $i$ occurs $n\mu_i$ times in $F$, we have $q_4=n\mu_i-(q_1+q_2+q_3)=n\mu_i-(a+b)\mu_i+q_1$. Since $\F \cup \{F\}$ is a set of MOFS, when we superimpose $F$ with each square $F_t \in \F$ we get the pair $(i,1)$ appearing precisely $\mu_i\lambda_{1,t}$ times, so $\mu_i \sum_{t=1}^k \lambda_{1,t}$ times in total. By the $\Z_w$-sum of $\F$, each $i$ in $Q_j$ of $F$ must contribute $x_j \bmod w$ to the total number of $(i,1)$ pairs. Hence, modulo $w$, we have
\begin{align*}
\mu_i \sum_{t=1}^k \lambda_{1,t}& \equiv x_1q_1+x_2q_2+x_3q_3+x_4q_4 \\
&\equiv x_1q_1 + x_2(a\mu_i-q_1)+x_3(b\mu_i-q_1)+x_4(n\mu_i-(a+b)\mu_i+q_1) \\
&\equiv q_1(x_1+x_4-x_2-x_3)+a\mu_i(x_2-x_4)+b\mu_i(x_3-x_4)+\mu_inx_4.
\end{align*}
Since $x_1+x_4-x_2-x_3\equiv 0 \bmod w$ we have shown the result.
\end{proof}

\begin{corollary}\label{cor-block-structure}
Let $\F=\{F_1,F_2,\dots,F_k\}$ be a set of binary {\rm MOFS} where $F_t$ has type $(n; \lambda_{0,t},\lambda_{1,t})$ for $1 \leq t \leq k$ and suppose some non-negative integers $w, x_1, x_2, x_3, x_4$ satisfy the hypotheses of \tref{thm-block-structure}. Let $F$ be a square of type $(n; \mu_0, \mu_1,\dots,\mu_{m-1})$ such that $\F \cup \{F\}$ is a set of $(k+1)$-{\rm MOFS}. If $\gcd(\mu_i,w)=1$ for some $i \in \{0,1,\dots,m-1\}$ then 
\begin{equation}\label{cor-(1,1)count}
\sum_{t=1}^k \lambda_{1,t}\equiv a(x_2-x_4)+b(x_3-x_4) + n x_4 \mod w.
\end{equation}
\end{corollary}

We now provide some examples which demonstrate the power of this last result. The first shows the maximality of a set of MOFS of odd order. This is notable because \tref{thm-rel-gen} shows that non-constant relations only apply to MOFS of even order.

\begin{example}\label{ex-p-rel}
Consider the following $2$-{\rm MOFS}$(3;2,1)$, along with its $\Z_3$-sum:
\[ 
\left[\begin{array}{cccccc}
1&0&0\\
0&1&0\\
0&0&1\end{array} \right]
+
\left[ \begin{array}{cccccc}
1&0&0\\
0&0&1\\
0&1&0\end{array} \right]
\equiv
\left[ \begin{array}{cccccc}
2&0&0\\
0&1&1\\
0&1&1\end{array} \right].
\]
Testing \eref{cor-(1,1)count} with $w=3$,
$\lambda_{1,1}=\lambda_{1,2}=a=b=x_4=1$ and $x_2=x_3=0$,
we find that this set of ${\rm MOFS}$ is maximal.
\end{example}

\begin{example}\label{pseudo-rel}\setlength{\arraycolsep}{2pt}
  Below, we give two sets of $2$-{\rm MOFS}$(6;4,2):$
\[ 
\left[\begin{array}{cccccc}
1&1&0&0&0&0\\
1&1&0&0&0&0\\
0&0&1&1&0&0\\
0&0&1&1&0&0\\
0&0&0&0&1&1 \\
0&0&0&0&1&1 \end{array}
\right]
+
\left[\begin{array}{cccccc}
1&1&0&0&0&0\\
1&1&0&0&0&0\\
0&0&0&0&1&1\\
0&0&0&0&1&1\\
0&0&1&1&0&0 \\
0&0&1&1&0&0\end{array} \right]
\equiv
\left[\begin{array}{cccccc}
1&1&0&0&0&0\\
1&1&0&0&0&0\\
0&0&1&1&0&0\\
0&0&1&0&1&0\\
0&0&0&1&0&1 \\
0&0&0&0&1&1 \end{array}
\right]
+
\left[\begin{array}{cccccc}
1&1&0&0&0&0\\
1&1&0&0&0&0\\
0&0&0&0&1&1\\
0&0&0&1&0&1\\
0&0&1&0&1&0 \\
0&0&1&1&0&0\end{array} \right]
\equiv
\left[\begin{array}{cccccc}
0&0&0&0&0&0\\
0&0&0&0&0&0\\
0&0&1&1&1&1\\
0&0&1&1&1&1\\
0&0&1&1&1&1 \\
0&0&1&1&1&1\end{array} \right].
\]
As shown, both have the same $\Z_2$-sum. Neither satisfies a relation, although in some sense they get close. Applying \Cref{cor-block-structure} with $w=3$, we see that neither of the above pairs can be extended to a triple with a square of type $(6;5,1)$ or $(6;4,2)$. In particular, both pairs are type-maximal. However, neither pair is maximal as both can be extended to a triple using, for example, the following square of type $(6;3,3):$
\[\left[\begin{array}{cccccc}
1&0&0&0&1&1\\
0&1&0&0&1&1\\
1&1&0&1&0&0\\
1&1&1&0&0&0\\
0&0&1&1&1&0\\
0&0&1&1&0&1\\
  \end{array} \right].
\]
\end{example}

We have seen that if a set of binary MOFS satisfies a non-constant full relation then its $\Z_2$-sum has the block structure (\ref{block-matrix}), which is a special case of the more general block structure (\ref{block-matrix3}). Hence, we can combine \tref{thm-block-structure} with the results in \sref{Relations} to get the following result.

\begin{theorem}\label{thm-max}
Let $\F=\{F_1,F_2,\dots,F_k\}$ be a set of binary {\rm MOFS} where $F_t$ has type $(n; \lambda_{0,t},\lambda_{1,t})$ for $1 \leq t \leq k$. Suppose $\F$ satisfies a non-constant full relation and let $F$ be a square of type $(n; \mu_0, \mu_1,\dots,\mu_{m-1})$. If  $\F \cup \{F\}$ is a set of $(k+1)$-{\rm MOFS} then
 $$\mu_i \sum_{t=1}^k \lambda_{0,t} \equiv \mu_i \sum_{t=1}^k \lambda_{1,t} \equiv 0 \mod 2,$$
 for $0 \leq i \leq m-1$.
\end{theorem}

\begin{proof}
  Suppose $\F$ satisfies a non-constant full $(a,b)$-relation. By \lref{lem-MOFS-2019}, $\F$ satisfies the hypothesis of \tref{thm-block-structure} with $x_1=x_4=0$, $x_2=x_3=1$ and $w=2$. Let $F$ be a frequency square of type $(n; \mu_0, \mu_1,\dots,\mu_{m-1})$ and suppose $\F \cup \{F\}$ is a set of $(k+1)$-MOFS. Then \tref{thm-block-structure} implies
$$\mu_i \sum_{t=1}^k \lambda_{1,t} \equiv \mu_i(a+b) \mod 2,$$
for $0 \leq i \leq m-1$.

Furthermore, by \tref{thm-rel-gen}(ii) both $a$ and $b$ have the same parity, and therefore 
$\mu_i\sum_{t=1}^k \lambda_{1,t}$ must be even. By \tref{thm-rel-gen}(i) $n$ is even. Hence, $\sum_{t=1}^k \lambda_{1,t}$ and $\sum_{t=1}^k \lambda_{0,t}$ have the same parity, by \lref{lm:parity}. Therefore $\mu_i\sum_{t=1}^k \lambda_{0,t}$ must also be even.
\end{proof}

In the case where a symbol in the new square has odd frequency, \tref{thm-max} implies the following corollary.

\begin{corollary}\label{cor-odd-max}
Let $\F=\{F_1,F_2,\dots,F_k\}$ be a set of binary {\rm MOFS} where $F_t$ has type $(n; \lambda_{0,t},\lambda_{1,t})$ for $1 \leq t \leq k$ and suppose $\F$ satisfies a non-constant full relation. Let $F$ be a square of type $(n; \mu_0, \mu_1,\dots,\mu_{m-1})$ such that $\F \cup \{F\}$ is a set of $(k+1)$-{\rm MOFS}. If $\mu_i$ is odd for some $i \in \{0,1,\dots,m-1\}$ then $$\sum_{t=1}^k \lambda_{0,t} \equiv \sum_{t=1}^k \lambda_{1,t} \equiv 0 \mod 2.$$
\end{corollary}

The contrapositive of Corollary \ref{cor-odd-max} can be used to show that under certain conditions, a set of MOFS can only be extended by squares with even frequencies. This is summarised in the following result.

\begin{theorem}\label{thm-max-gen}
Let $\F$ be a set of binary $k$-{\rm MOFS} that satisfies a non-constant full relation. Suppose there is some symbol $x \in \{0,1\}$ such that $x$ has odd frequency in at least one square in $\F$. If there exists a frequency square $F$ such that $\F \cup \{F\}$ is a set of $(k+1)$-{\rm MOFS}, then all symbols of $F$ must occur with even frequency.
\end{theorem}

\begin{proof}
Let $\F=\{F_1,F_2, \dots, F_k\}$, where $F_t$ has type $(n; \lambda_{0,t},\lambda_{1,t})$ for $1 \leq t \leq k$. Let $F$ be a frequency square of type $(n; \mu_0, \mu_1,\dots,\mu_{m-1})$ and suppose $\F \cup \{F\}$ is a set of $(k+1)$-MOFS. Furthermore, suppose there is a symbol $x \in \{0,1\}$ such that $x$ has odd frequency in at least one $F_i$. By \tref{thm-rel2}, $x$ must have odd frequency in an odd number of squares in $\F$, so $\sum_{t=1}^k \lambda_{x,t} \equiv 1 \bmod 2$. This contradicts the conclusion of Corollary \ref{cor-odd-max}, and thus $\mu_i$ must be even for all $0 \leq i \leq m-1$.
\end{proof}

\tref{thm-max-gen} implies the following result, which generalises \tref{max-MOFS-2019} to include sets of binary MOFS that are not necessarily balanced. 

\begin{corollary}\label{cor-max}
Suppose $k$, $\lambda_0$ and $\lambda_1$ are odd. Let $\F$ be a set of $k$-{\rm MOFS}$(n;\lambda_0,\lambda_1)$ that satisfies a full relation. Then $\F$ is type-maximal. 
\end{corollary}

\begin{proof}
  By Corollary \ref{cor-con-rel}  the relation that $\F$ satisfies must be non-constant. Suppose that $\F \cup \{F\}$ is a set of $(k+1)$-MOFS for some frequency square $F$. By \tref{thm-max-gen} all symbols of $F$ must have even frequency.
  Since $\lambda_0$ and $\lambda_1$ are odd, this means that $\F$ is type-maximal.
\end{proof}

It is clear in \Cref{cor-max} that we could have reached a stronger conclusion than type-maximality. If a set of MOFS satisfies the hypotheses of \tref{thm-max-gen}, then we can only extend the set using squares with even frequencies. Example \ref{notmax} demonstrates that the relation plays a crucial role here. That example shows a $5$-MOFS$(6;3,3)$ that is type-maximal but can be extended by squares with odd frequencies. 

Unfortunately, \Cref{cor-max} does not generalise to the case where $\lambda_0$ and $\lambda_1$ are even. This was shown in \cite{MOFS-2019} where the authors provide an example of a set of $3$-MOFS$(n;\lambda_0,\lambda_1)$ with $\lambda_1=\lambda_0\equiv 0 \bmod 2$ that satisfies a non-constant full relation but is not type-maximal. 

We finish this section with two examples of \tref{thm-max-gen} at work.

\begin{example}\label{oddmax1}
The following is a set of $5$-{\rm MOFS}$(6;3,3)$ from {\rm \cite{MOFS-2019}}. Its $\Z_2$-sum, given on the right, shows it satisfies a full $(5,3)$-relation. By \tref{thm-max-gen}, it can only be extended by squares with even frequencies and is therefore type-maximal.
\begin{equation*}
\left[ \begin{array}{cccccc}
11011&10111&01100&00001&00010&11100 \\
10100&01111&11011&00010&11100&00001\\
01111&11000&10111&11100&00001&00010\\
01001&10001&00101&10110&01110&11010\\
10010&00110&01010&01101&11001&10101\\
00100&01000&10000&11011&10111&01111\end{array} \right]
\hspace{1.2cm}
\left[ \begin{array}{cccccc}
0&0&0&1&1&1\\
0&0&0&1&1&1\\
0&0&0&1&1&1\\
0&0&0&1&1&1\\
0&0&0&1&1&1\\
1&1&1&0&0&0 \end{array} \right]
\end{equation*}
Although it is type-maximal, it is not maximal, since it is orthogonal to the following set of $5$-{\rm MOFS}$(6;4,2)$.
\[
\left[ \begin{array}{cccccc}
01000&10111&01001&10000&00010&00100 \\
00100&10000&01000&00101&00011&11010 \\
10011&01110&10000&00000&01001&00100 \\
10000&00000&00101&01010&11100&00011 \\
00001&01000&00110&00010&10100&11001 \\
01110&00001&10010&11101&00000&00000 \end{array} \right] \]
Together these give a maximal set of $10$-{\rm MOFS}$(6)$ containing squares of two types.
\end{example}

\begin{example}\label{oddmax2}
The following is a set of $9$-{\rm MOFS}$(6;3,3)$ given in {\rm \cite{MOFS-2019}}, and chosen to satisfy a full $(3,3)$-relation:
\begin{equation}\label{e:9MOFS633} 
\left[ \begin{array}{cccccc}
111000001&001101010&010110100&100110011&001011101&110001110\\
 010110100&111000001&001101010&110001110&100110011&001011101\\
 001101010&010110100&111000001&001011101&110001110&100110011\\
 111111111&000000111&100011000&010101001&011010010&101100100\\
 100011000&111111111&000000111&101100100&010101001&011010010\\
 000000111&100011000&111111111&011010010&101100100&010101001\end{array}
\right].
\end{equation}
By \tref{thm-max-gen}, it can only be extended using squares of type $(6;4,2)$ and is therefore type-maximal.
However it can be extended by the following set of $4$-{\rm MOFS}$(6;4,2)$:
\begin{equation}\label{e:4MOFS642}
\left[ \begin{array}{cccccc}
0000&0000&0100&0011&1001&1110\\
 0001&1100&0010&0001&1100&0010\\
 1101&1010&0011&0100&0000&0000\\
 1010&0111&1001&0100&0000&0000\\
 0110&0001&1000&1000&0010&0101\\
 0000&0000&0100&1010&0111&1001
 \end{array} \right].
\end{equation}
Together, \eref{e:9MOFS633} and \eref{e:4MOFS642} form a maximal set of $13$-{\rm MOFS}$(6)$.
\end{example}

\section{Computational results}\label{comp}

We know from \Cref{cor-gen-bound} that an upper bound on the cardinality of a set of binary MOFS$(n)$ is $(n-1)^2$. It is an open question how close we can get to these complete sets in general, especially when mixing types. In this section we report some computational data for small orders. For simplicity, throughout this section we will refer to binary frequency squares of type $(n;\lambda_0,\lambda_1)$ as {\it squares of type} $\lambda_1$, provided the order is clear from context. By complementing if need be, we may assume that $\lambda_1\le n/2$.
We say that $k$-MOFS$(n)$ are a $k$-MOFS$(n;\Lambda)$ where $\Lambda \subseteq \{1,2,\dots,\lfloor n/2 \rfloor\}$ is the set of types of the $k$ frequency squares in the set. We stress that this notation means that every type in $\Lambda$ must occur within the set of MOFS, and no type outside of $\Lambda$ can occur.
Furthermore, we define $f(n;\Lambda)$ to be the maximum $k$ such that there exists a binary $k$-MOFS$(n;\Lambda)$. We will refer to the overall maximum (over all $\Lambda$) as $f(n)$. 

Understanding $f(n)$ in general seems to be a very difficult problem. However, a modest first step is to investigate small orders. For the main computations reported in this section, we used the following algorithm. The input was $n,\ \Lambda$ and $\F$, a MOFS$(n)$. The output was $\F\,'$, a maximum MOFS$(n)$ containing $\cal{F}$. Each element of $\F\,'\setminus\F$ had to be of a type in $\Lambda$. However, $\F\,'$ did not have to be a MOFS$(n,\Lambda)$, for two reasons. Firstly, the squares in the input $\F$ may have types that are not in $\Lambda$. Secondly, not every type in $\Lambda$ was required to be present in the output $\F\,'$.

The first task for our algorithm was to generate all possible \emph{mates}, that is, all frequency squares that were orthogonal to every square in $\F$ and were of some type in $\Lambda$.  When generating mates of type 3 and order 6, we only included a square or its complement, not both. This is because complementing does not change orthogonality. Next, the computer constructed a graph $\Gamma_{\F}$ with the mates as its vertices, and edges indicating orthogonality. After that, it searched $\Gamma_{\F}$ for a maximum clique. The union of $\F$ and the set of mates corresponding to this maximum clique yields $\F\,'$, the algorithm's output.

For very small orders we could use $\F=\emptyset$ as the input. For
slightly larger orders we needed to precompute a catalogue of all
possible inputs $\F$ up to isomorphism, where the types of the squares
in $\F$ were specified in advance. For isomorphism screening we used nauty
\cite{nauty}. 
Two sets of MOFS are \emph{isomorphic} if one can be obtained from the
other by some sequence of the following operations:
\begin{itemize}
  \item Applying the same permutation to the rows of all squares in the set.
  \item Applying the same permutation to the columns of all squares in the set.
  \item Transposing all squares in the set.
  \item Permuting the symbols in one of the squares.
  \item Permuting the squares within the set (in cases where we have imposed an
    order on the set).
\end{itemize}

The case $n=2$ is trivial since any single square is a complete set. For $n=3$ all squares have type $1$. It was a simple task to establish that
$f(3)=2$ and there are no complete sets of binary {\rm MOFS}$(3)$.
For $n=4$ it was shown in \cite{MOFS-2019} that all sets of type-maximal MOFS$(4;\{2\})$ are complete, so $f(4)=f(4;\{2\})=9$. We also considered squares of type $1$ and we found that $f(4;\{1\})=3$ and $f(4;\{1,2\})=7$.

For $n=5$ we have two possible types of squares; $1$ and $2$. We found that
$f(5)=8$ and there are no complete sets of binary {\rm MOFS}$(5)$.
There are 2160 squares of order 5 (120 of type $1$ and 2040 of type $2$) so testing for orthogonality amongst these squares was computationally easy. We first tested each type individually and found that $f(5;\{1\})=5$
and $f(5;\{2\})=8$. Mixing types did not result in any larger MOFS, in fact $f(5;\{1,2\})=8$. 

\begin{example}
The following is a maximal set of binary $8$-{\rm MOFS}$(5)$ containing two squares of type $1$ and six of type $2$:
\begin{equation*}
\left[ \begin{array}{ccccc}
00000000&00010110&00000001&01111000&10101111\\
00100001&00001010&00011100&10010011&01100100\\
00111010&11000000&00100110&00000101&00011001\\
10001100&00110101&01001011&00100010&00010000\\
01010111&00101001&10110000&00001100&00000010
\end{array} \right].
\end{equation*}
\end{example}

For the case $n=6$ the computations became more difficult. There are three types of squares in this case; $1$, $2$ and $3$. The authors of \cite{MOFS-2019} found that $f(6;\{3\})=17$ and there are 18 unique sets of 17-MOFS$(6;\{3\})$ up to isomorphism, all of which satisfy a full $(3,3)$-relation. Our programs confirmed that $f(6;\{1\})=10$. Together with our other values of $f(n;\{1\})$, this was already known from literature on EPAs \cite{HCD}. This was a simple computation since, up to isomorphism, there is only one square of type $1$. Here is an example of a type-maximal $10$-MOFS$(6;\{1\})$:
\begin{equation}\label{Max_Case0.1} \setlength{\arraycolsep}{4pt}
\left[ \begin{array}{cccccc}
0110000100&1001000000&0000001001&0000110000&0000000010&0000000000\\
0000011000&0100100011&0000000000&0000000000&0001000100	&1010000000\\
0000000010&0000001100&0010100000&1100000000&0000000000	&0001010001\\
0000000000&0010010000&0101000000&0000000101&1000101000&0000000010\\
1000000001&0000000000&0000000000&0011001010&0100010000&0000100100\\
0001100000&0000000000&1000010110&0000000000&0010000001&0100001000\end{array} \right]
\end{equation}

We can represent a set of MOFS even more compactly by converting the entries in the superimposed form from binary into decimal. From this point on, most of our MOFS will be represented this way. The decimal representation of (\ref{Max_Case0.1}) is 
\begin{equation*}\label{Max_Case0.1_dec} 
\left[ \begin{array}{cccccc}
388&576&9&48&2&0\\
24&291&0	&0&68&640\\
2&12&160&768&0&81\\
0&144&320&5&552&2\\
513&0&0&202&272&36\\
96&0	&534&0&129&264\end{array} \right].
\end{equation*}

Our next task was to establish that $f(6;\{2\})=14$. Up to isomorphism, there are four squares of type $2$, and there are 683 pairs of MOFS$(6;\{2\})$ which we used as input. Interestingly, only two of the 683 pairs could not be extended to a triple (that is, had empty $\Gamma_{\F}$), and are therefore type-maximal. They are the $2$-MOFS$(6;\{2\})$ from Example \ref{pseudo-rel}.

Many of the $14$-MOFS$(6;\{2\})$ that we found could be shown to be type-maximal using \Cref{cor-block-structure}. Below is one such type-maximal $14$-MOFS$(6;\{2\})$, along with its $\Z_3$-sum:
 \begin{equation*}
\left[ \begin{array}{cccccc}
14883&9309&2052&1456&842&4224\\
13516&10930&5889&2112&17&302\\
17&3844&12600&8258&7306&741\\
3424&4224&8198&9097&4724&3099\\
406&4419&2281&6684&9248&9728\\
520&40&1746&5159&10629&14672 \end{array} \right]
\hspace{1cm}
\left[\begin{array}{cccccc}
1&1&2&2&2&2\\
1&1&2&2&2&2\\
2&2&0&0&0&0\\
2&2&0&0&0&0\\
2&2&0&0&0&0\\
2&2&0&0&0&0\end{array} \right].
\end{equation*}
 
We now know that a complete MOFS$(6)$ cannot be achieved using squares of a single type. We failed to extend any of the 18 type-maximal sets of $17$-MOFS$(6;\{3\})$ using squares of type $2$. Note that since each $17$-MOFS$(6;\{3\})$ satisfies a non-constant full relation, \tref{thm-max-gen} implies that they cannot be extended by squares of type $1$ or $3$. However, to rule out the existence of larger MOFS$(6)$ of mixed type, we had to do further tests.

\begin{table}[ht]
\centering
\begin{tabular}{ |l|l|l|l|l|l| }
 \hline
 Case & type 1 & type 2 & type 3 & mates &max. \\ 
 \hline
 1&2+&0&0&(93,96)&10\\
 2&0&2+&0&(0,7969)&14\\
 3&0&0&2+&(5937,7413)&17\\
 \hline
 4&1+&1+&0&(4113,5264)&14\\ 
 5&1+&0&1+&(8307,8997)&15\\ 
 6&0&1&1+&(0,9696)&14\\
 7&0&2&1+&(0,6528)&15\\
 8&0&3+&0+&(2201,10788)&15\\
 9&1&1&0+&(8358,9602)&14\\
 10&1&2+&0+&(2206,6499)&14\\
 11&2+&1+&0+&(3257,3934)&13\\
 \hline
\end{tabular}
\caption{}
\label{cases}
\end{table}

To this end, we divided the computation into cases as shown in Table~\ref{cases}. The first three cases have already been discussed, and the last eight cases involve MOFS of mixed type. Each case describes the type of the MOFS $\F$ which were used as input into our algorithm, as well as the types in $\Lambda$ that were allowed when constructing mates. The entry $x$ in the column titled {\it type t} indicates that $\F$ must contain precisely $x$ squares of type $t$, and when a ``+" also appears next to an entry, it indicates this type is included in $\Lambda$. For example, in Case 10 we found all non-isomorphic $3$-MOFS$(6;\{1,2\})$ with precisely one square of type 1 and two squares of type 2, and input each of these into our algorithm, allowing for mates of type 2 and 3. In every case $\F$ is either a pair or a triple.

The fifth column of Table~\ref{cases} indicates the minimum and maximum number of mates (vertices in $\Gamma_{\F}$) found amongst all input pairs or triples. For Case 3 these bounds were stated in \cite{MOFS-2019}. The cases in which the lower bound was 0 are worth remarking on. The two pairs in Case 2 that could not be extended to triples are type-maximal and were discussed in Example \ref{pseudo-rel}. In Case 6 there were 2668 non-isomorphic pairs, 16 of which had no mates, three of which satisfied a non-constant full relation. In Case 7 there were 4\,408\,252 non-isomorphic triples, 13\,613 of which had no mates. Of those, there were 1471 that satisfied a non-constant relation, 424 of which satisfied a non-constant full relation. In all these instances, \tref{thm-max-gen} implies that the set cannot be extended by a square of type 3, which explains why $\Gamma_{\F}$ was empty. Note that \tref{thm-max-gen} only applies directly when a relation is full. However, if a relation is not full then there is a proper subset of the squares which satisfy a full relation.

Next we give an example of a type-maximal MOFS$(6)$ for each of the eight cases of mixed type, namely Case 4 to Case 11. The first example is a type-maximal $14$-MOFS$(6;\{1,2\})$ (Case 4), the second is a type-maximal $15$-MOFS$(6;\{1,3\})$ (Case 5), and so on. Note that although the $14$-MOFS$(6;\{2,3\})$ that we provide for Case 6 is type-maximal, there are other examples that are not. The simplest way to obtain one is to remove a square of type 2 from the $15$-MOFS$(6;\{2,3\})$ that we provide for Case 7.
{\small  \setlength{\arraycolsep}{3pt}
\begin{equation*}
\begin{array}{cc}
\left[ \begin{array}{cccccc}
12440&8260&3110&2721&1361&778\\
11650&6185&8818&1564&0&453\\
0&3776&13061&9291&2364&178\\
1825&151&0&6482&8936&11276\\
111&9520&2504&0&5766&10769\\
2644&778&1177&8612&10243&5216
\end{array} \right]&
\left[ \begin{array}{cccccc}
31487&18352&16396&5470&7617&2595\\
19829&30982&5805&20169&58&5074\\
3088&21057&12138&22198&16847&6589\\
710&3231&21875&8609&23160&24332\\
20904&7402&3028&19227&13845&17511\\
5899&893&22675&6244&20390&25816
\end{array} \right]\\
\\
\left[ \begin{array}{cccccc}
15358&13386&11553&2700&1873&183\\
13719&6625&10770&8829&3950&1160\\
1144&9894&4900&6227&11677&11211\\
3591&795&7388&14249&10466&8564\\
10281&3829&463&9542&12944&7994\\
960&10524&9979&3506&4143&15941
\end{array} \right]&
\left[ \begin{array}{cccccc}
31351&20874&17504&12269&3350&665\\
24451&17397&6238&19640&10090&8197\\
18008&7469&13232&26627&16623&4054\\
1215&10482&19277&5457&30340&19242\\
292&28188&3747&4814&18897&30075\\
10696&1603&26015&17206&6713&23780
\end{array} \right] \end{array}
\end{equation*}
\\
\begin{equation*}
\begin{array}{cc}
\left[ \begin{array}{cccccc}
30232&15780&16742&722&41&2049\\
22645&5065&10240&8232&18178&1206\\
11042&16448&5291&4400&24725&3660\\
224&8222&4612&23555&3448&25505\\
271&17440&9841&19116&14530&4368\\
1152&2611&18840&9541&4644&28778
\end{array} \right]&
\left[ \begin{array}{cccccc}
16382&8617&2072&1636&3283&775\\
8979&5173&10445&3626&1992&2550\\
234&9298&8097&10887&893&3356\\
1693&3012&9574&6491&10800&1195\\
2336&3915&639&9628&4230&12017\\
3141&2750&1938&497&11567&12872
\end{array} \right]\\
\\
\left[ \begin{array}{cccccc}
7196&15106&0&57&230&1985\\
6369&0&3618&5508&283&8796\\
8626&2252&3409&4615&5672&0\\
0&4721&940&2378&11269&5266\\
1610&1332&12425&2704&4416&2087\\
773&1163&4182&9312&2704&6440
\end{array} \right]&
\left[ \begin{array}{cccccc}
7225&2724&23&234&976&321\\
2241&1170&2556&4352&635&549\\
925&2145&3906&179&4212&136\\
208&369&4768&1613&2183&2362\\
294&536&233&2132&1441&6867\\
610&4559&17&2993&2056&1268
\end{array} \right] \end{array}
\end{equation*}
}
Cases $4$ and $5$ directly imply $f(6;\{1,2\})=14$ and $f(6;\{1,3\})=15$, respectively. Cases $6$, $7$ and $8$ together imply that $f(6;\{2,3\})=15$ and cases $9$, $10$ and $11$ together imply that $f(6;\{1,2,3\})=14$. Hence the overall maximum is $17$ and is only achieved by balanced squares.

\begin{theorem}
$f(6)=17$ and there are no complete sets of binary {\rm MOFS}$(6)$.
\end{theorem}

For $n\le6$ we note that $f(n)=f(n;\{\lfloor n/2 \rfloor\})$. It would interesting to establish whether this holds more generally. However, note that $f(5)$ is shared by $f(5,\{2\})$ and $f(5,\{1,2\})$. 

  \let\oldthebibliography=\thebibliography
  \let\endoldthebibliography=\endthebibliography
  \renewenvironment{thebibliography}[1]{%
    \begin{oldthebibliography}{#1}%
      \setlength{\parskip}{0.4ex plus 0.1ex minus 0.1ex}%
      \setlength{\itemsep}{0.4ex plus 0.1ex minus 0.1ex}%
  }%
  {%
    \end{oldthebibliography}%
  }


\begin{thebibliography}{99}

\bibitem{MOFS-2019}
T. Britz, N.\,J. Cavenagh, A. Mammoliti and I.\,M. Wanless,
Mutually orthogonal binary frequency squares, 
{\it Electron. J. Combin.}, {\bf 27}(3) (2020), P3.7, 26pp.

\bibitem{MOFS-2020}
N.\,J. Cavenagh, A. Mammoliti and I.\,M. Wanless,
Maximal sets of mutually orthogonal frequency squares,
{\it Des.\ Codes Cryptogr.}, {\bf89} (2021), 525--558.

\bibitem{HCD}
  C.\,J. Colbourn and J.\,H. Dinitz (eds),
  {\it Handbook of Combinatorial Designs} (2nd
  ed.), Chapman \& Hall/CRC, Boca Raton, 2007.

\bibitem{JJ-2020}
J. Jedwab and T. Popatia,
A new representation of mutually orthogonal frequency squares, 
{\it J. Combin. Math. Combin. Comput.}, to appear.

\bibitem{nauty}
  B.\,D.~McKay, \texttt{nauty} graph isomorphism software, available at\\
  \url{http://cs.anu.edu.au/~bdm/nauty}.

\bibitem{LS}
D.\,Stinson, 
A short proof of the nonexistence of a pair of orthogonal Latin squares of order six, 
{\it J.~Combin.~Theory Ser.~A}, {\bf 36} (1984), 373--376.

\end{thebibliography}
\end{document}